\documentclass[a4paper, 12pt]{article}
\usepackage{graphicx}
\usepackage{amssymb,amsfonts,amsthm, amscd}
\usepackage{amsmath}
\usepackage{hyperref}
\usepackage[mathcal]{eucal}

\usepackage[left=3cm,right=3cm,
    top=3cm,bottom=3cm,bindingoffset=0cm]{geometry}

\newtheorem{theorem}{{\sc Theorem}}[section]

\newtheorem{lemma}{{\sc Lemma}}[section]

\newtheorem{remark}{{\sc Remark}}[section]
\newtheorem{definition}{{\sc Definition}}[section]

\newtheorem{case}{\sc{Case}}

\begin{document}

\author{Tigran Hakobyan}
\title{\textbf {On an asymptotic behavior of the divisor function $\tau(n)$}}
 \date{}
\maketitle

\begin{abstract}
For $\mu>0$ we study an asymptotic behavior of the sequence defined as $$T_{n}(\mu)=\frac{max_{1\leq m \leq {n^{\frac{1}{\mu}}}}\{\tau (n + m)\}}{\tau(n)},\ n=1,2,...$$ where $\tau(n)$ denotes the number of natural divisors of the given $n\in \mathbb{N}$. The motivation  of this observation is to explore  whether $\tau$  function oscillates rapidly in  small  neighborhoods of  natural  numbers.
\end{abstract}

\section*{Introduction}
Recall  that  the  function $\tau(n)$  defined as  the  number of  positive divisors of  the  given  positive  integer $n$  has  many  investigated asymptotic  properties  and   some  of  them  are  presented below.
\begin{enumerate}
\item $\forall \epsilon > 0$  $\tau(n)=o(n^{\epsilon}).[1]$\\
\item $\forall \epsilon > 0 \ \exists $ infinitely  many $n\in \mathbb{N} $ such that $$\tau(n)> 2^{(1-\epsilon)\frac{\ln(n)}{\ln(\ln(n))}}$$ and $$\tau(n)< 2^{(1+\epsilon)\frac{\ln(n)}{\ln(\ln(n))}}$$
  holds  for  sufficiently  large  $n$. ( Vigert, 1907)

\item $$\sum_{k=1}^n \tau(k)=\sum_{k=1}^n [\frac{n}{k}]=n\ln(n)+(2\gamma-1)n+O(n^{\frac{13}{40}+\epsilon}), \forall \epsilon > 0$$ where $\gamma$ is  the Euler's constant.[1]\\
\item  Worth mentioning  the result in [5]   concerning Karatsuba's problem on determining the asymptotic behavior of the sum $$S_{a}(x)=\sum_{n\le x}\frac{\tau(n)}{\tau(n+a)} $$ stated in 2004 which was estimated by M.A. Korolev  in 2010.
\end{enumerate}

\section{Basic assertions}

For $\mu>0$ consider the sequence $$T_{n}(\mu)=\frac{max_{1\leq m \leq {n^{\frac{1}{\mu}}}}\{\tau (n + m)\}}{\tau(n)},\ n=1,2,...$$
Let us assume that $(n_{k})$ is a sequence of positive integers such that $n_{k}={p_k}^{j_k}$ where $p_k$ is prime and $ j_{k}\in \mathbb{N}$ for all $k\in \mathbb{N}$.
\begin{definition}  {\small$$ \theta=\inf\{\lambda>0|\sum_{k=1}^N\tau(k)=N\ln(N)+(2\gamma-1)N+O(N^{\lambda+\epsilon}),  \forall\epsilon>0\}.$$}where $\gamma $ is the Euler's constant.\end{definition}
The main results of this paper are the following theorems.
\begin{theorem}
 If $\mu>0$, then $T_{n_k}(\mu)\to\infty$, as $j_k\to\infty$.
\end{theorem}

\begin{theorem}
 If $1\leq\mu<\theta^{-1}$, then $T_{n_k}(\mu)\to\infty$, as $n_k\to\infty$.
\end{theorem}
\section{Preliminary statements}
Obviously we may assume  that $ \mu \in \mathbb{N}$.\\Indeed, if theorem 1  holds  for  some $\mu_0>0$ then it holds for any $0<\mu<\mu_0$. On the other  hand theorem 2 follows  from theorem 1  as  we will  see  later.\\\\
Now we fix $ \mu \in \mathbb{N}$ , $ \mu\geq 2$  and suppose that $k=\mu m $ where $ m\in \mathbb{N}.$ 
\begin{definition} $$\nu_{p}(n)=\max\{k\geq0 :p^k|n\};$$  $$\Delta(n)=\sum_{\{p:p|n\}}\nu_{p}(n), n>1,$$

 and $$\Delta(1)=0.$$\end{definition}

Observe that using multiplicativity property of $\tau$ we will get  $$\tau(p^k+p^s)=\tau(p^s)\tau(p^{k-s}+1)=(s+1)\tau(p^{k-s}+1)$$ for any $s\in\{0,1,\dots,k-1\}$ and    prime number $p$. \\\\ On the other hand  $$\Delta(mn)=\Delta(m)+\Delta(n)$$ and  consequently $\Delta(n^k)=k\Delta(n)$ for every $m,n,k\in \mathbb{N}$. \newline 
\begin{lemma} If $(k-s)$ is  odd  then
$$\tau(p^{k-s}+1)\geq\tau(k-s)\geq\Delta(k-s).$$ 
\end{lemma}

\begin{proof}
Indeed,  if $a$  is  odd  and $a\vdots b$  then $$ (m^a+1)\vdots(m^b+1)$$ for  any $m\in \mathbb{N}$ and hence $$\tau(m^a+1)\geq \tau(a).$$ The latter inequality follows from $$\tau(n)=\prod_{\{p:p|n\}}(1+\nu_p(n))>\sum_{\{p:p|n\}}\nu_p(n)=\Delta(n)$$
\end{proof}.
\begin{definition}

Now  define $$A(k)=\sum_{s=1}^m(s+1)\Delta(k-s)$$ and
  $$A'(k)=\prod_{s=1}^m(k-s)^{s+1}.$$
	
	\end{definition}
	
	So $$A'(k)=\prod_{s=1}^m(\mu m-s)^{s+1}=\frac{(\mu m-1)!}{((\mu-1)m-1)!}\prod_{s=1}^m\frac{(\mu m-s)!}{((\mu-1)m-1)!}.$$\newline Notice that $$A(k)=\Delta(A'(k))$$ So $$A(k)=\Delta(\frac{(\mu m-1)!}{((\mu-1)m-1)!}\prod_{s=1}^m\frac{(\mu m-s)!}{((\mu-1)m-1)!})=$$$$=\Delta(\frac{(\mu m-1)!}{((\mu-1)m-1)!})+\Delta(\prod_{s=1}^m\frac{(\mu m-s)!}{((\mu-1)m-1)!})$$\newline
\begin{definition}Define  $$B(k)=\Delta(\frac{(\mu m-1)!}{((\mu-1)m-1)!})$$ and $$C(k)=\Delta(\prod_{s=1}^m\frac{(\mu m-s)!}{((\mu-1)m-1)!})$$\end{definition}  So  $$A(k)=B(k)+C(k)$$.

\begin{lemma}
There is  some constant $\gamma>0$ such that $$\Delta(k!)\leq\gamma k\ln(\ln(k))$$ holds for  any  $k\in \mathbb{N}.$
\end{lemma}

\begin{proof} By the  famous identity $\nu_{p}(n!)=\sum_{s=1}^\infty[\frac{n}{p^s}]$ (henceforth $[x]$ stands for integer part of $x\in R$)  we  obtain  that
$$\Delta(k!)=\sum_{p\leq k}\nu_{p}(k!)=\sum_{p\leq k}\sum_{s=1}^\infty[\frac{k}{p^s}]<\sum_{p\leq k}\frac{k}{p-1}=$$$$=k(\sum_{p\leq k}\frac{1}{p}+\sum_{p\leq k}\frac{1}{p(p-1)})<k(\sum_{p\leq k}\frac{1}{p}+\epsilon)<$$$$<k(\ln(\ln(k))+\delta)<\gamma k\ln(ln(k))$$ where $\epsilon=\sum_{p}\frac{1}{p(p-1)}>0$  and $\delta ,\gamma>0 $ .The lemma is proved.
\end{proof}

\begin{lemma}
There  exists a constant $c>0$ such that $$A(k)\geq cm^2\ln(ln(m))$$ for all $\ m\geq2 , m\in \mathbb{N}$.(  recall that $k=\mu m$)
\end{lemma}

\begin{proof} Notice that $A(k)=B(k)+C(k)$ and $B(k)=O(m\ln(\ln(m)))=o(m^2\ln(\ln(m)))$ by lemma 1. Now we estimate  $C(k)$.

$$C(k)=\sum_{l=(\mu-1)m}^{\mu m-1}\sum_{p\leq\mu m-1}(\sum_{s=1}^\infty[\frac{l}{p^s}]-\sum_{s=1}^\infty[\frac{(\mu-1)m-1}{p^s}]))\geq$$  $$\geq\sum_{l=(\mu-1)m}^{\mu m-1}\sum_{p\leq\mu m-1}(\sum_{p\leq\mu m-1}(\sum_{s=1}^\infty[\frac{l-(\mu-1)m+1}{p^s}])).$$

Furthermore $$\nu_{p}(n!)=\sum_{1\leq s\leq[\log_{p}(n)]}[\frac{n}{p^s}]>\frac{n}{p}(1+\frac{1}{p}+\ldots+\frac{1}{p^{[\log_{p}(n)]-1}})-[\log_{p}(n)]\geq$$$$\geq\frac{n-p}{p-1}-\log_{p}(n)=\frac{n-1}{p-1}-\log_{p}(np).$$
Thus $$C(k)\geq\sum_{l=(\mu-1)m}^{\mu m-1}(\sum_{p\leq\mu m-1}(\frac{l-(\mu-1)m}{p-1}-\log_{p}((l-(\mu-1)m+1)p)))=$$$$=
\sum_{p\leq\mu m-1}(\frac{1}{p-1}+\frac{2}{p-1}+\ldots+\frac{m-1}{p-1}-\log_{p}(2p\cdot3p\cdot\ldots\cdot mp))=$$$$=\frac{m(m-1)}{2}\sum_{p\leq\mu m-1}\frac{1}{p-1}-(m-1)\pi(\mu m-1)-\sum_{p\leq\mu m-1}\log_{p}(m!)$$
\newline 
\begin{definition} $$X(m)=\frac{m(m-1)}{2}\sum_{p\leq\mu m-1}\frac{1}{p-1},$$ $$Y(m)=(m-1)\pi(\mu m-1)$$ and $$Z(m)=\sum_{p\leq\mu m-1}\log_{p}(m!)$$\end{definition}.
Recall  that the functions $\pi(n),\frac{n}{\ln(n)} $and$\  Li(n)=\int_{2}^{n}\frac{dt}{\ln(t)}$ are equivalent as $n\rightarrow\infty$, where  \\
$\pi(n)=card\{1\leq k\leq n |k \ is\  prime\}$ for every $n\in \mathbb{N}.$(see$[1]$) \\ \ \\
From $$\pi(\mu m-1)=O(\frac{m}{\ln(m)})$$ we infer that$$Y(m)=(m-1)\pi(\mu m-1)=O(\frac{m^2}{\ln(m)})=o(m^2\ln(\ln(m))).$$
On the other hand  $Z(m)=\ln(m!)\sum_{p\leq\mu m-1}\frac{1}{\ln(p)}.$ 
Observe  that $$\sum_{p\leq\mu m-1}\frac{1}{\ln(p)}\leq L\sum_{s=2}^{\pi(\mu m-1)}\frac{1}{\ln(s\ln(s))}<L_{1}\sum_{s=2}^{\pi(\mu m-1)}\frac{1}{\ln(s)} $$ (since there  is an $\alpha>0$ such that$\ p_k>\alpha k\ln(k)\ $for every $k\in \mathbb{N}$ where $p_k $ is the $k$-th prime) and  that $$\sum_{s=2}^{\pi(\mu m-1)}\frac{1}{\ln(s)}\sim\int_{2}^{\pi(\mu m-1)}\frac{dt}{\ln(t)}\sim\pi(\pi(\mu m-1))\sim$$$$\sim\frac{\frac{\mu m-1}{\ln(\mu m-1)}}{\ln(\frac{\mu m-1}{\ln(\mu m-1)})}\sim\mu\frac{m}{{\ln(m)}^2}\ $$\newline (We say $f(x)\sim g(x)$ as $x\rightarrow \infty$ if there are positive constants $\alpha$ and $\beta$ such that $\alpha |f(x)|<|g(x)|<\beta |f(x)|$ for all sufficiently large $x\in \mathbb{R}$)  
\newline \newline     
Therefore using Stirling's formula in the form $$\ln(m!)=O(m\ln(m))$$ we will get  that $$Z(m)=O(\frac{m}{{\ln(m)}^2})\cdot O(m\ln(m))=O(\frac{m^2}{\ln(m)})=o(m^2\ln(\ln(m))).$$\\ \ \\
To  estimate $X(m)$ we  use  the  fact  that $$\sum_{\{p\leq n |p\ is\ prime\}}\frac{1}{p}>\ln(\ln(n))-1$$ for all $n\geq 2, n\in \mathbb{N}$, which exactly means that for all $c\in(0,\frac{1}{2}), \  X(m)$ and consequently $C(k)$ has  the property $C(k)\geq cm^2\ln(\ln(m))$ eventually. Hence $$A(k)=A(\mu m)\geq cm^2\ln(\ln(m))$$ eventually, as desired.\end{proof}

Thus $$A(k)=\sum_{s=1}^m(s+1)\Delta(k-s)\geq cm^2\ln(\ln(m))$$ for  $m$  large enough.\newline \newline It follows that $\exists s_{0}\in\{1,2,\ldots,m\}$ such that $$(s_{0}+1)\Delta(k-s_{0})\geq c\ mln(\ln(m)).$$ If$\ (k-s_{0})$ is odd then $$\tau(p^k+p^{s_{0}})\geq (s_{0}+1)\Delta(k-s_{0})\geq c\ mln(\ln(m))$$ (see  the  section "preliminary statements")\newline   hence $$\frac{\max_{1 \leq m \leq \sqrt[\mu]{n}}\tau(n+m)}{\tau(n)}\geq\frac{cm\ln(\ln(m))}{\mu m+1}>\frac{c}{2\mu}\ln(\ln(m))$$ for $m$ large  enough. 
\begin{remark}
Unfortunately ,it  may happen that $(k-s_0)$ is  even . To fix this  we  proceed  in the following  way .
\end{remark}
\begin{definition}

For an arbitrary $m\in \mathbb{N}$ and $\beta>0$ let us define $$I(m,\beta)=\sum_{\{1\leq s\leq m|\nu_{2}(k-s)>\beta ln(ln(m))\}}(s+1)\Delta(k-s).$$
\end{definition}

\begin{lemma}For every $\beta>0$, $$I(m,\beta)=o(m^2\ln(\ln(m)))$$ as $m\rightarrow\infty$.\end{lemma}
{\begin{proof} Suppose $k-s=2^la$,\ where $a$ is odd and $l>\beta\ln(\ln(m))$.\newline So $$a<\frac{\mu m}{2^{\beta\ln(ln(m))}}$$.\\
Define $$L(m)=\frac{\mu m}{2^{\beta\ln(ln(m))}}=\frac{\mu m}{{\ln(m)}^{\beta\ln(2)}}.$$\\
Since $$(k-s)\in\{(\mu-1)m,(\mu-1)m+1,\ldots,\mu m-1\},$$ one has that  for fixed $a$ there is at most one value of $l$ such that
$$2^la\in\{(\mu-1)m,(\mu-1)m+1,\ldots,\mu m-1\},$$hence there are at most $$L^{*}(m)\leq L(m)$$ summands with $$\nu_{2}(k-s)>\beta \ln(\ln(m)).$$
Let us number them ,say $$s_{1},s_{2},\ldots,s_{L^{*}(m)}$$ and write $$k-s_{j}=2^{l_{j}}a_{j}$$, where $a_{j}$ is odd and $$l_{j}>\beta\ln(\ln(m))$$ is integer for every $j\in\{1,2,\ldots,L^{*}(m)\}.$\\ Observe that $$\Delta(k-s)=l_{j}+\Delta(a_{j}).$$Hence if we define $$I\equiv\sum_{j=1}^{L^{*}(m)}(s_{j}+1)\Delta(k-s_{j})$$ we will get that $$I=\sum_{j=1}^{L^{*}(m)}(s_{j}+1)l_{j}+\sum_{j=1}^{L^{*}(m)}(s_{j}+1)\Delta(a_{j}).$$
\begin{definition}
Let us define $$I_1=\sum_{j=1}^{L^{*}(m)}(s_{j}+1)l_{j}$$ and $$I_2=\sum_{j=1}^{L^{*}(m)}(s_{j}+1)\Delta(a_{j}).$$
\end{definition}
Let us estimate $I_1$. \begin{definition} Define $$T_{m}=\{[\beta\ln(\ln(m))]+1,[\beta\ln(\ln(m))]+2,\ldots,[\log_{2}(\mu m)]\} $$ for $m$ large enough.\end{definition} It is clear that $$l_{j}\in T_{m}$$ for every $$j\in\{1,2,\ldots,L^{*}(m)\}.$$

Let us fix $$t\in T_{m}$$ and consider those $s$ for which $$k-s=2^ta$$, where $a$ is odd.\newline Since $a$ takes values from a progression with difference $d=2$, it follows that $2^ta$ takes values from a progression with difference $d=2^{t+1}$.Consequently the corresponding sum $$S(m,d)=\sum_{\{\nu_{2}(k-s)=t\}}(k-s)\leq\sum_{l=0}^{\epsilon+1}((\mu-1)m+ld)$$where $$\epsilon d\leq m<(\epsilon+1)d.$$
 Hence $$S(m,d)\leq(\epsilon+1)(\mu-1)m+\frac{\epsilon(\epsilon+1)}{2}d+\{(\mu-1)m+(\epsilon+1)d\}\leq$$$$\leq(\mu-1)\frac{m^2}{d}+(\mu-1)m+\frac{m}{2}(\epsilon+1)\leq$$$$\leq(\mu-1)\frac{m^2}{d}+(\mu-1)m+\frac{m}{2}(\frac{m}{d}+1)+\{(\mu-1)m+(\epsilon+1)d\}\leq$$$$\leq(\mu-1)\frac{m^2}{d}+(\mu-1)m+\frac{m}{2}(\frac{m}{d}+1)+\{(\mu-1)m+(\frac{m}{d}+1)d\}=$$$$=(\mu-\frac{1}{2})\frac{m^2}{d}+(2\mu-\frac{1}{2})m+d\leq(\mu-\frac{1}{2})\frac{m^2}{d}+(2\mu+\frac{1}{2})m.$$ So $$\sum_{l_{j}=t}s_{j}l_{j}\leq t((\mu-\frac{1}{2})\frac{m^2}{2^{t+1}}+(2\mu+\frac{1}{2})m)$$ for all $t\in T_{m}$.\\
Thereby $$I_1=\sum_{j=1}^{L^{*}(m)}s_{j}l_{j}+\sum_{j=1}^{L^{*}(m)}l_{j}=\sum_{t\in T_{m}}\sum_{l_{j}=t}s_{j}l_{j}+\sum_{j=1}^{L^{*}(m)}l_{j}\leq$$$$\leq\sum_{t\in T_{m}}t((\mu-\frac{1}{2})\frac{m^2}{2^{t+1}}+(2\mu+\frac{1}{2})m)+L(m)\log_{2}(m)=$$$$=(\mu-\frac{1}{2})m^2\sum_{t\in T_{m}}\frac{t}{2^{t+1}}+(2\mu+\frac{1}{2})m\sum_{t\in T_{m}}t+\frac{\mu m}{{\ln(m)}^{\beta \ln(2)}}\cdot\log_{2}(m)\leq$$$$\leq(\mu-\frac{1}{2})m^2\theta_{m}+(2\mu+\frac{1}{2})m({\log_{2}(m)}^2)+\frac{\mu m}{{\ln(m)}^{\beta \ln(2)}}\cdot\log_2{m}=$$$$=o(m^2)=o(m^2\ln(\ln(m)))$$ since $\theta_{m}=\sum_{t\in T_{m}}\frac{t}{2^{t+1}}\rightarrow 0.$\\\\\ %as $m\rightarrow\infty.$

Let us estimate $I_2$. \newline \newline One has that 
$$I_2=\sum_{j=1}^{L^{*}(m)}(s_{j}+1)\Delta(a_{j})\leq\sum_{j=1}^{L^{*}(m)}(m+1)\Delta(a_{j})\leq(m+1)\sum_{j=1}^{[L(m)]}\Delta(j)\leq$$$$\leq\gamma(m+1)L(m)\ln(\ln(L(m)))$$\newline by Lemma 1 and the fact that $a_{j}\leq [L(m)]$ for all $$j\in\{1,2,\ldots,L^{*}(m)\}.$$ \\ According to equality $L(m)=\frac{2m}{{\ln(m)}^{\beta \ln(2)}}$ we will get  $$I_2\leq\gamma(m+1)\frac{2m}{{\ln(m)}^{\beta \ln(2)}}\ln(\ln(\frac{2m}{\ln(m)}))\leq$$$$\leq C\frac{m^2\ln(\ln(m))}{{\ln(m)}^{\beta \ln(2)}}=o(m^2\ln(\ln(m)))$$
In the long run $I=I_1+I_2=o(m^2\ln(\ln(m)))$ as $m\rightarrow\infty$. The lemma is proved.

\end{proof}

\begin{lemma}There are  $c>0$ and $\beta>0$ such that for all sufficiently large $m\in\mathbb{N}$ it is always possible to select an  $$s_0\in{\{1,2,\ldots,m\}}$$ such that $$\nu_2(k-s_0)\leq\beta\ln(\ln(m))$$ and $$(s_0+1)\Delta(k-s_0)\geq c m\ln(\ln(m))$$ . \end{lemma}
\begin{proof}
\begin{definition}
Define $$I^{*}(m,\beta)=\sum_{\{1\leq s\leq m|\nu_{2}(k-s)\leq\beta\ln(\ln(m))\}}(s+1)\Delta(k-s).$$\end{definition}
So $$I^{*}(m,\beta)=A(k)-I(m,\beta)$$
According to lemmas 3 and  4  there exist  $c>0$ and $\beta>0$ such that the inequality $$I^{*}\geq cm^2\ln(\ln(m))$$ holds for sufficiently large $m$. So there is always $\ s_{0}\in\{1,2,\ldots,m\}$
such that $$\nu_{2}(k-s_{0})\leq\beta\ln(\ln(m))$$ and $$(s_{0}+1)\Delta(k-s_{0})\geq cm\ln(\ln(m))$$ for large $m$. The lemma is proved.
\end{proof}

\section{Proof of the theorem 1}

In  accordance with  lemma 4 $$\Delta(k-s_{0})\geq\frac{cm\ln(\ln(m))}{s_0+1}\geq\frac{cm\ln(\ln(m))}{m+1}\geq\frac{c}{2}\ln(\ln(m))$$ for sufficiently large $m$.

\bigskip
If $(k-s_{0})$ is odd then we are done, since  $$\tau(p^{k-s_0}+1)\geq\tau(k-s_0)\geq\Delta(k-s_0) .$$
Let $k-s_{0}=2^ta$, where $a$ is odd, $t\geq 1$. \newline \newline Then $a<\frac{k}{2}=\frac{\mu m}{2}.$ Let $\beta=\frac{c}{4}>0.$\newline

 Now we have that $$\Delta(a)= \Delta(k-s_0)-t\geq\frac{c}{2 }\ln(\ln(m))-\beta \ln(\ln(m))=\frac{c}{4}\ln(\ln(m))$$
when $m$ is large enough.\\ \medskip

\textbf{Consider 3 cases}:
\newline \newline
\begin{case} $a$ has a prime factor $q>2\mu$.\end{case}
\begin{proof}In this case $a=q\cdot b$, where $b$ is odd thus $b=\frac{a}{q}<\frac{\mu m}{2q}<\frac{m}{4}$. Hence $\exists$ an odd $r$ such that
$rb \in \{(\mu-1) m,(\mu-1) m+1,\dots,\mu m -1\}.$
 Let us take the smallest such $
 r $ and let $s^*$ satisfies $$ \mu m - s^* = k - s^* = rb$$, consequently $$ s^* = \mu m -rb \geq \mu m - ((\mu-1)m+2b)=m-2b>\frac{m}{2}.$$

 By the way $$\Delta (b) = \Delta(a) -1>\frac{c}{8}\ln(\ln(m)) $$ ($m$ is sufficiently large).
 Hence $$(s^*+1)\Delta(k-s^*)=(s^*+1)\Delta(rb)>s^*\Delta(b)>$$$$>\frac{m}{2}\cdot\frac{c}{4}\ln(\ln(m))=\frac{c}{8}m\ln(\ln(m)).$$
 We have that  $$k-s^*=rb$$ is odd and  so  we are done in this case.
\end{proof}
 \begin{case} All prime factors of $a$ are less then $2\mu$.\end{case} 
 \begin{proof} To  see  what   is  going  on  in  this  case   let  us  write   the   canonical  factorisation  of  $a$.
 Suppose  $a=p_1^{\alpha_1}p_2^{\alpha_2}\cdot\ldots\cdot p_t^{\alpha_t}$ where $p_j$ is an odd prime less than $2\mu$ and $\alpha_j$ is a positive integer for every $j\in \{1,2,\dots,t\}$.
  Assume that $\alpha_1\leq\alpha_2\leq\ldots\leq \alpha_t .$Thereby $$ \alpha_t\geq\frac{\Delta(a)}{t}\geq\frac{\Delta(a)}{\pi(2\mu)}\geq\frac{c}{4\pi(2\mu)}\ln(\ln(m)).$$
  Let  us  take  $$b=p_t^{\beta_t}$$   and  impose the  following  conditions on it .\newline
	$1) b<\frac{m}{4}$ \\ $2)\beta_{t}\geq\rho\ln(\ln(m))$ for some  $\rho>0$.\\ \newline
  To  satisfy  the  first   condition  it  is  enough  to   find $\beta_t$ from $$(2\mu)^{\beta_t}<\frac{m}{4},$$
  or $$\beta_t< (log_{2\mu}(\frac{m}{4})),$$ so  it  is  enough   to  take $$\beta_{t} = [\frac{c}{4\pi(2\mu)}\ln(\ln(m))]$$ for large $m$, to   satisfy both conditions.
  To  finish   the  proof   we  need  only  to  repeat   the  last   part  of    solution  of   case 1.
  \end{proof}
	
  \begin{case} $a = 1$.\end{case}
  \begin{proof} In this case  $$ k - s_0 = 2^t a = 2^t $$ so $$ t \geq log_2((\mu - 1)m) > \beta \ln(\ln(m)) $$ for large $m$, 
  which is  a contradiction. Case 3 is proved.\end{proof}

	Notice that we have proved the theorem for $$k = \mu m$$ only. So we need to prove it in any case.
\begin{proof}
  Assume that $$ k = \mu m - r$$ for sufficiently large $m$ and $$r \in \{1,2,\dots, \mu - 1\}.$$
  Therefore $$ \tau (p^k + p^{s-r}) = \tau (p^{\mu m - r} + p^{s-r})= \frac{s - r + 1}{s + 1} \tau (p^{\mu m} + p^s).$$
   It is evident that we may assume $$s>2\mu.$$ \newline Indeed, in the former summation  first $2\mu$ summands do not influence upon the  sum, since their sum is $$o(m^2 \ln(\ln(m))).$$
	
  So, if we take an $s_0$ which maximizes $$ \tau(p^{\mu m} + p^s) $$ we will get that $$ \tau (p^{\mu m - r} + p^{s_0 - r}) \geq \frac{1}{2}\tau (p^{\mu m} + p^{s_0}),$$ since $$ \frac{s-r+1}{s+1} \geq \frac{1}{2}$$ for $ s > 2\mu $ and $r<\mu.$\newline  
  Observing that $$ \mu (s_0 - r) \leq \mu m - \mu r < \mu m - r $$ we are done . 
	The theorem is now proved.\end{proof}
	\section{Proof of the theorem 2}
	\begin{lemma} If $1\leq \mu < {\theta}^{-1}$, then there is a constant $c>0$ such that $$T_{N}(\mu)>c\frac{ln(N)}{\tau(N)}$$ for all positive integers $N$.\end{lemma}

\begin{proof} Using the formula  $$\sum_{k=1}^N \tau(k)=N\ln(N)+(2\gamma-1)N+O(N^{\theta+\epsilon})$$ we will get that  $$\sum_{k=N+1}^{N+[\sqrt[\mu]{N}]} \tau(k)=(N+\sqrt[\mu]{N})\ln(N+\sqrt[\mu]{N})+(2\gamma-1)(N+\sqrt[\mu]{N})+$$$$+O({(N+\sqrt[\mu]{N})}^{\theta+\epsilon})-(N\ln(N)+(2\gamma-1)N+O(N^{\theta+\epsilon})) =$$$$=N^{\frac{1}{\mu}}\ln(N)+O(N^{\frac{1}{\mu}})+O(N^{\theta+\epsilon})=N^{\frac{1}{\mu}}\ln(N)+O(N^{\frac{1}{\mu}})$$, since $\frac{1}{\mu}>\theta$ and $\epsilon>0$ is arbitrary.
	Thus $$\max_{1\leq m \leq  [\sqrt[\mu]{N}]}\{\tau (N + m)\}>\frac{\sum_{k=N+1}^{N+[\sqrt[\mu]{N}]} \tau(k)}{\sqrt[\mu]{N}}=$$$$=\frac{N^{\frac{1}{\mu}}\ln(N)+O(N^{\frac{1}{\mu}})}{\sqrt[\mu]{N}}>c\ln(N),$$ for some $c>0$ and any positive integer $N$.\newline  It follows that $$T_{N}(\mu)=\frac{\max_{1\leq m \leq [\sqrt[\mu]{N}]}\{\tau (N + m)\}}{\tau(N)}>c\frac{\ln(N)}{\tau(N)}$$ for all positive integers $N$.
	 The lemma is proved.\end{proof}

Now  we prove theorem $2$.  
\begin{proof}Suppose $(n_{k})$ is a  sequence of  positive  integers  such  that $n_k={p_k}^{j_k}$, where  $p_k$  is  prime  , $j_k$ is  a positive integer for   each $k\in \mathbb{N}$ and $n_k\rightarrow\infty$ as $k\rightarrow\infty$. \newline \newline
Suppose $E>0$ is an arbitrary number. According to theorem $1$ there  is an $A>0$ such that $j_k>A$ implies $T_{n_k}(\mu)>E.$ \newline \newline 
Lemma $6$ shows  that $$ T_{n_k}(\mu)>c\frac{\ln(n_k)}{\tau(n_k)}=c\frac{j_k\ln(p_k)}{j_k+1} \geq \frac{1}{2}c\ln(p_k)$$ So there is $B>0$ such that $p_k>B$ implies $T_{n_k}(\mu)>E$. \newline \newline 
The condition $n_k\rightarrow\infty$ shows  that there are only finitely many $k\in N$ with $j_k\leq A$ and $p_k\leq B$. So there is a positive integer $k=k(E)$ such that $T_{n_k}(\mu)>E$ for any positive integer $k>k(E)$. \newline 
Since $E $ was  arbitrary, we conclude that $T_{n_k}(\mu)\rightarrow\infty$ as $k\rightarrow\infty$. \newline 
The  theorem is proved. 
\end{proof}

\bigskip
\begin{remark}
In the end worth mentioning that for every $\mu>0$ the relation  $$\lim_{n\to\infty}T_n(\mu)=\infty$$ seems to be plausible.
\end{remark}

\end{document}